\documentclass[12pt]{elsarticle}

\usepackage{lineno,hyperref}
\usepackage{dsfont}
\usepackage{amssymb}
\usepackage{latexsym}
\usepackage{amsmath}
\usepackage{color}
\usepackage{comment}
\usepackage{amsthm}
\usepackage{enumerate}
\usepackage{bm}
\usepackage{bbm} 
\usepackage[bbgreekl]{mathbbol} 
\usepackage{setspace} 
\newif\ifrs
\rstrue
\ifrs \usepackage{mathrsfs} \fi  
\newif\ifcol
\coltrue 

\newtheorem{theorem}{Theorem}[section]

\newtheorem{lemma}[theorem]{Lemma}

\newtheorem{remark}[theorem]{Remark}
\newtheorem{example}[theorem]{Example}
\numberwithin{equation}{section}
\newtheorem{theorem*}{Theorem}
\newtheorem{ass*}[theorem*]{Assumption}
\newtheorem{note*}[theorem*]{Note}
\newtheorem{lemma*}[theorem*]{Lemma}
\newtheorem{definition*}[theorem*]{Definition}
\newtheorem{proposition*}[theorem*]{Proposition}
\newtheorem{corollary*}[theorem*]{Corollary}
\newtheorem{remark*}[theorem*]{Remark}
\newtheorem{example*}[theorem*]{Example}
\numberwithin{equation}{section}
\newif\ifcol
\coltrue
\ifcol
\newcommand{\colorr}{\color{black}}

\newcommand{\colorn}{\color[rgb]{1,1,1}}
\newcommand{\colorred}{\color[rgb]{0.8,0,0}}

\else

\newcommand{\colorr}{\color{black}}

\newcommand{\colorn}{\color{black}}
\newcommand{\colorred}{\color{black}}

\fi
\excludecomment{en-text}
\includecomment{jp-text}
\includecomment{comment}

\def\bd{\begin{description}}
\def\ed{\end{description}}

\def\D2{\bbD_{2,\infty-}}

\def\D{{\bf D}}

\def\R{{\bf R}}

\def\yeq{\>=\>}

\def\ygeq{\>\geq\>}

\def\ep{\epsilon}
\def\half{\frac{1}{2}}

\def\down{\downarrow}

\def\halflineskip{\vspace*{3mm}}
\def\nn{\nonumber}
\def\be{\begin{equation}}
\def\ee{\end{equation}}
\def\bea{\begin{eqnarray}}
\def\eea{\end{eqnarray}}
\def\beas{\begin{eqnarray*}}
\def\eeas{\end{eqnarray*}}
\def\bi{\begin{itemize}}
\def\ei{\end{itemize}}

\def\bd{\begin{description}}
\def\ed{\end{description}}
\def\l{\left}
\def\r{\right}

\newcommand{\bbD}{{\mathbb D}}

\newcommand{\bbI}{{\mathbb I}}

\newcommand{\bbN}{{\mathbb N}}

\newcommand{\bbR}{{\mathbb R}}

\newcommand{\bbZ}{{\mathbb Z}}

\modulolinenumbers[5]

\begin{document}

\begin{frontmatter}

\title{Divergence of an integral of a process with small ball estimate\tnoteref{mytitlenote}}
\tnotetext[mytitlenote]{This work was in part supported by
Japan Science and Technology Agency CREST JPMJCR14D7;
Japan Society for the Promotion of Science Grants-in-Aid for Scientific Research
No. 17H01702 (Scientific Research); and by a Cooperative Research Program of the Institute of Statistical Mathematics.}

\author[tar]{Yuliya Mishura}
\ead{myus@univ.kiev.ua}

\author[gra,jpn]{Nakahiro Yoshida}
\ead{nakahiro@ms.u-tokyo.ac.jp}

\address[tar]{Taras Shevchenko National University of Kyiv}
\address[gra]{Graduate School of Mathematical Sciences, University of Tokyo}
\address[jpn]{Japan Science and Technology Agency CREST}

\begin{abstract}
The paper contains sufficient conditions on the function $f$ and the stochastic process $X$ that supply the rate of divergence of the integral functional $\int_0^Tf(X_t)^2dt$ at the rate $T^{1-\ep}$
as $T\to\infty$ for every $\ep>0$. These conditions include so called small ball estimates which are discussed in detail. Statistical applications are provided.
\end{abstract}


\end{frontmatter}


\section{Introduction}
The problem of convergence or divergence of     perpetual   integral functionals $$\int_0^\infty g(X(t))\,{\mathrm{d}t}$$ for several classes of stochastic   processes and several classes of functions $g$ appears when studying a variety of issues. Let  $X=\{X(t), t\ge 0\}$ be a one-dimensional stochastic process with continuous   trajectories, and let $g:\R\rightarrow \R$ be a   continuous   function. Then for any $T>0$ the integral functional
\begin{equation}\label{integral-functional}
\int_0^T g\big(X(t)\big)\,{\mathrm{d}t}
\end{equation}
is defined. However, its properties and asymptotic behavior as $T\rightarrow \infty$ depend crucially on the properties of the process $X$ and of function $g$. The asymptotic behavior of the integral functional  $\int_0^T g(X(t))\,{\mathrm{d}t}$ is very different  even for one-dimensional Markov processes and depends on their transient or recurrent properties. On the one hand, conditions of existence of the perpetual integral functionals and in the case of non-existence, the rate of divergence,  were studied,  in the stochastic framework,  for various classes of one-  and multidimensional semimartingale and only partially for non-semimartingale stochastic processes, e.g, in   \cite{Ben-Ari, Erickson, Khoshnevisan, KMS, KoMiSi,  KuKuMi1, KuKuMi,  SaYor}. In the papers where the rate of the divergence was studied, corresponding normalizing factors were suitable for the central or other functional limit theorems. On the other hand, this question   arises in the parametric statistical estimation because such integral functionals appear as the denominators of the residual terms, see, e.g., \cite{KMM} for fractional models and \cite{KuMiRa} both for Wiener diffusions and fractional diffusions. In such case we need  the divergence of this integral with some fixed rate and with probability 1, in order to get strongly consistent estimators.  In this connection, the aim of the present  paper is to investigate the rate of convergence of the integral $\int_0^T f(X_t)^2dt$ to infinity as $T\rightarrow \infty$ depending on the properties of the measurable function $f$ and stochastic process $X$. It is well-known that in the case when $X$ is stationary, ergodic and $\mathbb{E} f^2(X_0)$ is finite, then $T^{-1}\int_0^T f(X_t)^2dt\rightarrow \mathbb{E} f^2(X_0)$
as $T\rightarrow \infty$, and then the  rate of divergence is evident, $T^{-1+\epsilon}\int_0^T f(X_t)^2dt\rightarrow \infty$ for any $\epsilon>0$. If process $X$ is not ergodic, the situation is more involved and the conditions on $f$ and   $X$ should be much more complicated. In our approach, these conditions include so called small ball estimates. Since these conditions are interesting both   themselves and from the point of view of various applications, see e.g. \cite{MS}, we consider them even in those examples where the processes are ergodic, see examples \ref{superexa1} and \ref{superexa2}.

The paper is organized as follows. Section \ref{sec_cond} contains the basic conditions for the function $f$ and for the  process $X$. Conditions for $X$ include small ball estimates which are presented in two versions, more mild and then more strong, with various examples. In Section \ref{sec_div} the main divergence theorem is proved, and then two  modifications and several examples are provided. Section \ref{stat_app} presents two statistical applications. Section \label{appn} contains some auxiliary results.

\section{Main conditions. Discussion of small ball estimates}\label{sec_cond}
Let $X=\{X_t, t \geq 0\}$ be a real-valued stochastic process, and $f=f(x): \R \rightarrow \R$ be a measurable function.
Let the following assumption hold:

 (A)  For any $T>0$ the integral
\beas
\bbI_T  &=& \int_0^T f(X_t)^2dt
\eeas

is correctly defined. Our   goal is to establish sufficient conditions on function $f$ and process $X$ that supply the divergence to infinity:

\begin{equation}\label{infdiv}
T^{-1+\ep}\>\bbI_T \to  \infty\quad a.s.
\end{equation}
as $T\to\infty$ for every $\ep>0$.

Since we have two objects involving into the problem: function $f$ and stochastic process $X$, and respective assumptions will be non-trivial, let us consider them separately, with comments and examples.
\subsection{Assumptions on function $f$}
Concerning function $f$, introduce the following notations: denote  the sets  $H_+(x,\eta)=[x,x+\eta)$ and $H_-(x,\eta)=(x-\eta,x]$. Basic assumptions on function $f$ will be as follows.

\begin{itemize}
\item[(A1)]
\begin{itemize}
\item[(i)]
There exist positive constants $K$ and $\eta_*$ such that
\beas
K(\eta):=\inf_{x\in\R}\min_{{\sf H}\in\{H_+,H_-\}}
\sup_{y\in {\sf H}(x,\eta)}\big|f(y)\big| &\geq& \eta^K
\eeas
for every $\eta\in(0,\eta_*)$.

\item[(ii)] Function $f$ is from  class $C^1(\R)$ and for some constant $C_0$ we have that
\beas
|f' (x)| &\leq& C_0(1+|x|)^{C_0}\qquad(x\in\R).
\eeas
\end{itemize}
\end{itemize}

Let us consider the equivalent form  of assumption $(A1), (i)$,   sufficient condition for it and the simplest examples and a counterexample.
\begin{lemma}
\begin{itemize}
\item[(1)] Assumption (A1), (i) is equivalent to any of the following conditions:
\begin{itemize}
\item[(A3)]
There exist positive constants $K,C$ and $\eta_*$ such that
\beas   K_1(\eta):=\inf_{x\in\R}\min_{{\sf H}\in\{H_+,H_-\}}
\sup_{y\in {\sf H}(x,\eta)}\big|f(y)\big|&\geq& C\eta^K \eeas
for every $\eta\in(0,\eta_*)$.
\item[(A4)] There exist positive constants $K$ and $\eta_*$ such that
\beas K_2(\eta):=\inf_{x\in\R}
\sup_{y\in  (x, x+\eta)}\big|f(y)\big| &\geq& \eta^K
\eeas
for every $\eta\in(0,\eta_*)$.
\item[(A5)]  There exist positive constants $K, C$ and $\eta_*$ such that

\beas
K_3(\eta):=\inf_{x\in\R}
\sup_{y\in  (x, x+\eta)}\big|f(y)\big| &\geq& C\eta^K
\eeas
for every $\eta\in(0,\eta_*)$.  \end{itemize}
\item[(2)] Let there exist such positive constants $\eta_0, \delta$ and  $d\in \mathbb{N}$   such that $f\in C^{(d)}(\R)$ and
    \begin{equation}\label{ineq-5} \inf_{x\in \R}\max_{0\le i\le d}\inf_{y\in [x, x+\eta_0]}\big|f^{(i)}(y)\big|  \geq  \delta.\end{equation} Then assumption $(A1), (i)$ holds.\end{itemize}\end{lemma}
    \begin{proof} \begin{itemize}
    \item[(1)] Let us prove equivalence of conditions $(A1), (i)$  and $(A3)$. Indeed, if $(A1), (i)$ holds then  $(A3)$ holds with $C=1$. Inversely, if $(A3)$ holds with $C\ge 1$ then $(A1), (i)$ obviously holds, and if $(A3)$ holds with $C< 1$, we can take $K'=K+1$ and $\eta_{*}^{'}=\eta_*\wedge C$. Equivalence of $(A4)$ and $(A5)$ can be established similarly.  Now, let us prove equivalence of $(A1), (i)$ and
    $(A4)$.   Indeed, let $(A1), (i)$ hold. Since $K_2(\eta)\ge K(\eta/2)$, then $K_2(\eta)\ge 2^{-K}\eta^K$, whence $(A5)$ consequently $(A4)$ holds.    Inversely, let $(A4)$ hold. Then both values $\inf_{x\in\R}
\sup_{y\in {\sf H_-}(x,\eta)}\big|f(y)\big| \ge K_2(\eta) \geq  \eta^K$ and $\inf_{x\in\R}
\sup_{y\in {\sf H_+}(x,\eta)}\big|f(y)\big| \ge K_2(\eta) \geq \eta^K$, whence $K(\eta)\ge  \eta^K$.
\item[(2)] Let assumption \eqref{ineq-5} hold. Let us fix $x\in\R$. Without loss of generality, assume that $$\inf_{y\in (x, x+\eta_0)}\big|f^{(d)}(y)\big|  \geq  \delta.$$
    If, additionally, $\inf_{y\in (x, x+\eta_0)}\big|f^{(d-1)}(y)\big|  \geq  \frac{\delta \eta}{2^3}, $ then we proceed with $f^{(d-2)}$. If $\inf_{y\in (x, x+\eta_0)}\big|f^{(d-1)}(y)\big|  <  \frac{\delta \eta_0}{2^3}, $ then we check in which of four   intervals $[x, x+\eta_0/4], [x+\eta_0/4, x+\eta_0/2], [x+\eta_0/2, x+3\eta_0/4]$ or $[x+3\eta_0/4, x+ \eta_0]$ there exists a point $y$ satisfying inequality  $\big|f^{(d-1)}(y)\big|  <  \frac{\delta \eta_0}{2^3}.$ Let, for example, $y\in [x+\eta_0/4, x+\eta_0/2].$ Then for any $z\in [x+3\eta_0/4, x+ \eta_0]$ we have that $|f^{(d-1)}(z)-f^{(d-1)}(y)|\ge \frac{\delta\eta_0}{4},$ therefore for any $z\in [x+3\eta_0/4, x+ \eta_0]$ we have that $|f^{(d-1)}(z)|\ge \frac{\delta \eta_0}{2^3}.$ Then we  continue the same way with $|f^{(d-1)}|$, and in the worst  case, the smallest value that we can obtain, is: $|f(z)|\ge \frac{\delta \eta_0^d}{2^{3d}}$ for $z$ in some interval of the diameter  $\frac{\eta_0^d}{2^{2d}}$. However, even this worst case means that we can put in assumption  $\eta_*=\frac{  \eta_0}{2^{2d}}$, $K=d$ and $C=\frac{\delta  }{2^{d}}$ in assumption $(A5)$, which is equivalent to $(A1), (i)$, as it was already established. So, the proof follows.
\end{itemize}\end{proof}
\begin{example} Consider the classes of functions satisfying assumption \eqref{ineq-5}. Obviously, any polynomial function $P_m(x)$ of $m$th power satisfies \eqref{ineq-5}  because at least one of its derivatives is a non-zero constant. Also, any linear combination of the form $\sum_{i=1}^k (a_i\sin(\alpha_i x)+b_i\cos(\beta_i x))$ satisfies this assumption as well as the rational function $\frac{P_m(x)}{Q_n(x)}$ with $m>n$
and $Q_n(x)\not=0$.
An example of $f$ that satisfies $(A1)$ is
\beas
f(x) &=& 1_{\{x\not=0\}}x^3\sin\l(\frac{1}{x}\r).
\eeas
A periodic version
\beas
f(x) &=& 1_{\{x\not\in \pi\bbZ\}}\big(\sin x\big)^3\times\sin\l(\frac{1}{\sin x}\r)
\eeas
is also an example that satisfies $(A1)$ and has infinitely many
clusters of null points in every neighborhood of $\infty$. Exponential function $e^x$ does not satisfy this assumption around $-\infty$.    \end{example}
\subsection{Assumptions on process  $X$, with examples}
The first group of assumptions describes the processes bounded in   $L^{\infty-}$. It is formulated as follows.

\begin{itemize}
\item[ (A2)]
\begin{itemize}
\item[(i)] (H\"{o}lder continuity in $L^{\infty-}$)
$X$ is continuous a.s. and
there exits a positive constant $\rho$ such that
\beas
\sup_{s,t\in\R_+:\>s<t<s+1}\frac{\big\|X_t-X_s\big\|_r}{|t-s|^{\rho}} &<& \infty
\eeas
for every $r>1$.
\item[(ii)] (boundedness in $L^{\infty-}$) $\sup_{t\in\R_+}\|X_t\|_r<\infty$ for every $r>1$.
 \item[(iii)] (relaxed small ball estimate)
 There exist positive constants { $\Delta_*$,} $\gamma$, $\lambda$, $\mu$, $K_1$, $K_2$ and $K_3$,   such that
\beas
\sup_{s\geq 0} \mathbb{P}\bigg[\sup_{t\in[s,s+\Delta]}
\big|X_t-X_s\big|\leq\eta\bigg]
&\leq&
K_1\exp\bigg(-K_2\eta^{-\lambda}\Delta^\mu\bigg)
\eeas
for all  $\Delta{ \in(0,\Delta_*)}$  and $\eta{\in(0, K_3\Delta^\gamma )}$.
 \end{itemize}
\end{itemize}
Why are we considering so complicated ``relaxed small ball estimate'' $ (A2), (iii)$? The reason is that a wide class of processes satisfies this, albeit a little more complicated, but milder condition, while a simpler but more rigid analogue, condition  $ (A2), (iv)$ is satisfied by a narrower one.  However, we will discuss both conditions.

\begin{example}\label{ex-pos}
Consider the class of processes satisfying assumption $ (A2), (iii)$ (relaxed small ball estimate).
In order to do this, let us combine Theorem~4.4 from \cite{LS} with assumptions from \cite{MS}.
More precisely, let $X = \left\{X_t,t\ge0\right\}$ be a centered Gaussian process. We assume now that  its variance distance satisfies two-sided power bounds: there exist $H\in(0,1]$, and $C_1,C_2,C_3>0$ such that for any $s,t\ge0$, $|t-s|\le C_3$ we have that
\begin{equation}\label{eq:(*)}
C_1|t-s|^{2H} \le \mathbb{E} \left(X_t-X_s\right)^2 \le C_2|t-s|^{2H}.
\end{equation}
Let us work within this   assumption. Note that Theorem~4.4 \cite{LS}, in a little bit adapted form, states the following: let
$\left\{Z_t,t\in[0,\Delta]\right\}$
be a centered Gaussian process. Then for any $0<a\le 1/2$ and $\eta>0$
\[
\mathbb{P}\left\{ \sup_{0\le t\le\Delta} \left|Z_t\right| \le \eta \right\} \le \exp\left\{ - \frac{\eta^4}{16a^2 \sum_{2\le i,j \le 1/a} (E\xi_i\xi_j)^2 } \right\},
\]
provided that
$a\sum_{2\le i\le 1/a} E\xi_i^2 \ge 32\eta^2$,
where $\xi_i=Z_{ia\Delta} - Z_{(i-1)a\Delta}$.
Now let us fix $s\ge0$, $\Delta>0$, and put $Z_t = X_{t+s} - X_s, \;0\le t\le\Delta$. Then
\[
\xi_i = Z_{ia\Delta} - Z_{(i-1)a\Delta} = X_{ia\Delta} - X_{(i-1)a\Delta},
\]
and it follows from assumption \eqref{eq:(*)} that
\[
C_1(a\Delta)^{2H} \le\mathbb{ E}\xi_i^2 \le C_2(a\Delta)^{2H},
\]
and so the inequality
$a\sum_{2\le i\le 1/a} \mathbb{E}\xi_i^2 \ge 32\eta^2$
is fulfilled if
$C_1(a\Delta)^{2H} \ge 32\eta^2$, or, that is the same,
\begin{equation}\label{eq:a} a\ge \left(\frac{4\sqrt{2}}{\sqrt{C_1}}\right)^{1/H}\frac{\eta^{1/H}}{\Delta}.\end{equation}
Together with the inequality $a\leq 1/2$ we get that $$\eta\leq C_4\Delta^H,$$
where $C_4=\frac{\sqrt{C_1}}{2^{2+H}\sqrt{2}}.$
Additionally, we assume that the increments of $X$ are positively correlated, more exactly, for any $s_i,t_i\in\R^+$, $i=1,2$, $s_1 \le t_1 \le s_2 \le t_2$
\begin{equation}\label{eq:(**)}
\mathbb{E}\left(X_{t_1}-X_{s_1}\right)\left(X_{t_2}-X_{s_2}\right) \ge0.
\end{equation}
Note that positive correlation immediately implies that
\begin{equation}\label{positive}
\sum_{2 \le i,j \le \frac1a} (\mathbb{E}\xi_i\xi_j)^2
\le \max_{2 \le i,j \le \frac1a}\mathbb{ E}\xi_i\xi_j \mathbb{E}(X_{\Delta+s} - X_s)^2
\le \max_{2 \le i  \le \frac1a} \mathbb{E} \xi_i^2 C_2 \Delta^{2H} \le C_2^2a^{2H} \Delta^{4H}.
\end{equation}
Now put $\Delta_*=C_3$, $\eta_*=1$, $\Delta\le\Delta_*$, $\eta\le C_4\Delta^H$,
$a=C_5\frac{\eta^{1/H}}{\Delta}$, where $C_5=\left(\frac{4\sqrt{2}}{\sqrt{C_1}}\right)^{1/H}$.
Then
\begin{equation}\nonumber
\begin{gathered}
\frac{\eta^4}{16a^2\sum_{2 \le i,j \le \frac1a} (\mathbb{E}\xi_i\xi_j)^2} \ge \frac{\eta^4}{16C_2^2   a^{2+2H} \Delta^{4H}}  \\
\ge \frac{\eta^4}{16C_2^2 C_5^{2+2H} \eta^{2/H+2}\Delta^{2H-2}}
=C_6\frac{\Delta^{2-2H}}{\eta^{2/H-2}},
\end{gathered}
\end{equation}

where $C_6=\frac{1}{16C_2^2 C_6^{2+2H}}$. It means that assumption $(A2), (iii)$ holds with $ \Delta_*=C_3$, $\eta_*=1$, $K_1=1$, $K_2=C_6$, $K_3=C_4$, $\gamma=H$, $\mu=2-2H$, $\lambda=\frac2H-2$.

Evidently, assumptions \eqref{eq:(*)} and \eqref{eq:(**)} hold for fractional Brownian motion with $H>\frac12$.
According to \cite{BGT} and \cite{MS}, a subfractional Brownian motion with $H>\frac12$ also satisfies \eqref{eq:(*)} and \eqref{eq:(**)}. Note, however, that   assumption $(A2), (ii)$ fails for these processes.
\end{example}
 The next example supplies us with four classes of the processes   satisfying  all assumptions $(A2), (i)-(iii)$.
\begin{example}\label{superexa}(Periodic Brownian bridge)
  Consider a process that in some sense is a periodic Brownian bridge. Namely, let $X^{(k)}=\{X^{(k)}_t, t\in[k, k+1)\}$ be a sequence of independent   Brownian bridges, constructed between the points $(k,0)$ and $(k+1,0), k\ge 0$. They satisfy the relation of a form $$X^{(k)}_t=(k+1-t)\int_k^t\frac{dW^{(k)}_u}{k+1-u}, t\in[k, k+1),$$
where $W^{(k)}, k\ge 0$ is a sequence of independent Wiener processes, and let $X_t=X^{(k)}_t, t\in[k, k+1)\}.$ Evidently, we constructed a Gaussian process, and simple calculations show that its characteristics equal
$$\mathbb{E}X_t=0, \; \mathbb{E}\left(X^{(k)}_t\right)^2=(t-k)(k+1-t),\; \mathbb{E}\left(X^{(k)}_t-X^{(k)}_s\right)^2=(t-s)(1+s-t).$$
The middle equality means that   assumption $(A2),  (ii)$ holds, while   last equality means that for $s, t\in [k, k+1)$ and $0\le t-s\le 1/2$ we have that
$$\frac{t-s}{2}\le \mathbb{E}\left(X_t-X_s\right)^2\le t-s.$$
Consider now $t$ and $s$ from neighbor intervals, and let $s\in [k-1,k), t\in [k,k+1)$ and $t-s<1/2$. Then, on the one hand, we have the following relations:
$$\mathbb{E}\left(X_t-X_s\right)^2=\mathbb{E}\left(X^{(k)}_t\right)^2+\mathbb{E}\left(X^{(k-1)}_s\right)^2=t-s-(t-k)^2-(s-k)^2\leq t-s.$$
On the other hand, for $s\leq k\leq t$ we have that $(t-k)^2+(s-k)^2\leq (t-s)^2$, and for $t-s<1/2$ we have the inequality
$$\mathbb{E}\left(X_t-X_s\right)^2\geq t-s-(t-s)^2\geq \frac{t-s}{2}.$$
In particular, it means that $$\mathbb{E}\xi_i^2\ge \frac{a\Delta}{2}\ge 32\eta^2$$ consequently the inequality
$a\sum_{2\le i\le 1/a} \mathbb{E}\xi_i^2 \ge 32\eta^2$ holds provided that $a<1/2$ and $\eta\le \frac{\Delta^{1/2}}{8\sqrt{2}}$. All the relations above supply assumption  $(A2), (i)$ and relations \eqref{eq:(*)} with $\rho=H=1/2$. Note that the increments of $X$ are not positively, but negatively correlated. Consider only interval $[0,1]$, other cases can be treated similarly. On this interval, it is easy to see that for any $0\leq s\leq t\leq u\leq v\leq 1$
$$\mathbb{E}\left(X_t-X_s\right)\left(X_v-X_u\right)=-(v-u)(t-s)<0.$$

In view on negative correlation of increments, we can not apply upper bound \eqref{positive}. However, we can  calculate and evaluate the sum $S:=\sum_{2 \le i,j \le \frac1a} (\mathbb{E}\xi_i\xi_j)^2$ explicitly:
\begin{equation}\begin{gathered}\label{equ: for neg}S=\sum_{2 \le i  \le \frac1a} (\mathbb{E}(\xi_i )^2)^2+\sum_{2 \le i,j \le \frac1a, i\neq j} (\mathbb{E}\xi_i\xi_j)^2 =\left(\frac{1}{a}-1\right)\left(a\Delta(1-a\Delta)\right)^2\\+\left(\left(\frac{1}{a}-1\right)^2
 -\left(\frac{1}{a}-1\right)\right)a^4\Delta^4\le a\Delta^2+a^2\Delta^4.
\end{gathered}\end{equation}
Furthermore, $a<1/2, \Delta<1$, therefore, $S<2a\Delta^2$.
Therefore, taking into account  \eqref{eq:a} with $H=1/2$ and considering $a=C_5\frac{\eta^2}{\Delta}$ with $\eta\le C_4\Delta^{1/2}$, we get
\begin{equation}\begin{gathered}\label{equ: for S}
\frac{\eta^4}{16a^2\sum_{2 \le i,j \le \frac1a} (\mathbb{E}\xi_i\xi_j)^2}=\frac{\eta^4}{16a^2S}  \ge \frac{\eta^4}{32a^3  \Delta^2  } \ge \frac{\Delta}{32C_5^3 \eta^2}.
\end{gathered}\end{equation} It means that assumption $(A2), (iii)$ holds with   $\Delta_*=1, K_1=1, K_2=\frac{1}{32C_5^3}, K_3=C_4, \gamma=1/2, \mu=1, \lambda=2.$
\end{example}
\begin{example}\label{superexa1}(Stationary Ornstein--Uhlenbeck process) Consider even more simple and natural example. Having in mind calculations from Example \ref{superexa}, we can omit some technical details. So, introduce   a stationary Ornstein--Uhlenbeck process of the form $$X_t=\int_{-\infty}^te^{\theta(s-t)}dW_s,$$
where $W$ is a two-sided Wiener process, $\theta>0$. For the technical simplicity, we put $\theta=1.$   Then $\mathbb{E}X_t^2=\frac{1}{2}$, and this process is Gaussian, therefore condition $(A2), (ii)$ holds, $$X_t-X_s=\left(e^{-t}-e^{-s}\right)\int_{-\infty}^se^{z}dW_z+e^{-t} \int_{s}^te^{z}dW_z,$$
for any $s<t$, whence $$\mathbb{E}\left(X_t-X_s\right)^2=1-e^{s-t}.$$
Evidently,on the one hand, $1-e^{s-t}\le t-s$. On the other hand,  we can state  that $1-e^{-x}=e^{-\theta}( t-s)\ge e^{-1}( t-s)$ if  $t-s<1$ (here $\theta\in (-x,0)$). Therefore two-sided inequality \eqref{eq:(*)} holds with $H=1/2$,  and assumption $(A2),  (i)$ holds. In addition, Moreover, for any $s\le t\le u\le v$ we have that
$$\mathbb{E}\left(X_t-X_s\right)\left(X_v-X_u\right)=\frac12\left(e^t-e^s\right)\left(e^{-v}-e^{-u}\right)<0.$$
So, the increments are negatively correlated. Let us evaluate the sum $S$ from   \eqref{equ: for neg}, taking into account that if we choose $\Delta_*=2$ and $a=1/2$, then $a\Delta<1$:
  \begin{equation}\begin{gathered}\label{equ: for neg1}S=\sum_{2 \le i  \le \frac1a} (\mathbb{E}(\xi_i )^2)^2+\sum_{2 \le i,j \le \frac1a, i\neq j} (\mathbb{E}\xi_i\xi_j)^2 =\left(\frac{1}{a}-1\right)\left(1-e^{-a\Delta}\right)^2\\+\frac{1}{2}\left(1-e^{-a\Delta}\right)^2\left(1-e^{a\Delta}\right)^2
  \sum_{2 \le i<j \le \frac1a}e^{-2(j-i)a\Delta}\\ \le \frac{1}{a}a^2\Delta^2+\frac{1 }{2}ea^4\Delta^4\frac{1}{a^2}\le \left(1+\frac{e}{2}\right)a\Delta^2.
  \end{gathered}\end{equation}  Since we got the same upper bound as in the Example \ref{superexa}, up to a constant multiplier,  we can make the same conclusions.
\end{example}
\begin{example}\label{superexa2}(stationary fractional Ornstein-Uhlenbeck process)  Let $H \in (1/2,1)$, and let $B^H=\lbrace B_t^H, t \in \mathbb{R} \rbrace$ be a two-sided fractional Brownian motion with Hurst index $H$, that is, a centered Gaussian process with covariance function
\begin{equation}\nonumber
\mathbb{E} B_t^H B_s^H = \frac12 (|t|^{2H}+ |s|^{2H}- |t-s|^{2H}).
\end{equation}
Let $- \infty \le a < b \le + \infty$, and let a measurable function $h: [a,b] \rightarrow \mathbb{R}$ satisfy assumption
\begin{equation}\nonumber
\int_{[a,b]^2} |h(u)| |h(v)| |u-v|^{2H-2}dudv < \infty.
\end{equation}
Then the integral $\int_{[a,b]} h(z) dB_z^H$ is correctly defined and is a Gaussian random variable with zero mean and variance
\begin{equation}\label{const: C_H}
C_H \int_{[a,b]^2}  h(u) h(v) |u-v|^{2H-2} dudv,\; C_H=H(2H-1).
\end{equation}
In this connection, we can introduce a   fractional Ornstein-Uhlenbeck process $X_t= \int_{-\infty} ^t e^{\theta(s-t)} dB_s^H, $ $\theta>0$, that is a Gaussian stationary process with zero mean. For the technical simplicity, we put $\theta=1.$  Let's calculate some of its quadratic moment characteristics, in order to evaluate the left-hand side of \eqref{positive}. This evaluation includes several  constants whose value is not important, therefore we denote by $C$ various constants whose value can change from line to line and even inside the same line. First, according to \eqref{const: C_H},
\begin{equation}\label{I0}
\begin{gathered}
\mathbb{E}  X_t^2=C_H e^{-2t} \int_{(-\infty,t]^2} e^{u+v} |u-v|^{2H-2} du dv
=\left( u'=-u, v'=-v \right)\\
=C_H e^{-2t} \int_{[-t,\infty)^2} e^{-u'-v'} |u'-v'|^{2H-2} du'dv'\\
=\left(u'+t=z, v'+t=w \right)=C_H \int_{\mathbb{R}_+^2} e^{-z-w} |z-w|^{2H-2}dzdw=:C_HI_0.
\end{gathered}
\end{equation}
Taking into account Gaussian property of $X$, we conclude that  assumption    $(A2),  (ii)$ holds. Further, for any $0 \le s \le t$ applying suitable change of variables
\begin{equation}\nonumber
X_t-X_s=(e^{-t}-e^{-s}) \int_{-\infty}^s e^z dB_z^H+ e^{-t} \int_s^t e^z dB_z^H,
\end{equation}
whence
\begin{equation}\label{differ}
\begin{gathered}
\mathbb{E}  (X_t-X_s)^2=C_H (e^{-t} - e^{-s})^2 \int_{(-\infty,s]^2} e^{u+v} |u-v|^{2H-2}du dv\\
+2C_H e^{-t} (e^{-t} - e^{-s}) \int_{-\infty}^s \int_s^t e^{u+v} |u-v|^{2H-2} dudv
+ C_H e^{-2t} \int_{[s,t]^2} e^{u+v} |u-v|^{2H-2}dudv\\=C_H\Big(\left(e^{s-t}-1\right)^2I_0+2e^{s-t}\left(e^{s-t}-1\right)
\int_{0}^\infty\int_0^{ t-s} e^{-u+v} |u-v|^{2H-2}du dv\\+e^{2s-2t}\int_0^{ t-s}\int_0^{ t-s} e^{u+v} |u-v|^{2H-2}du dv\Big).
\end{gathered}
\end{equation}
Now, on the one hand, taking  into account the above relations, the fact that $0\le 1-e^{-x}\le x$ for $x>0$  and evident relation $e^{s-t}-1<0$, we can get   that for any $0\le s\le t$
\begin{equation}\label{burbur1}
\begin{gathered}
0 \le \mathbb{E}   \left(X_{t} - X_{s} \right)^2
\le C_H\left(  (t-s)^2 I_0+   \int_{[0, t-s]^2} |z-w|^{2H-2} dzdw\right)\\ \le C\left( (t-s)^2 I_0 + (t-s)^{2H}\right).
\end{gathered}
\end{equation}
On the other hand, $2H<2$, therefore $\frac{\left(e^{s-t}-1\right)^2}{(t-s)^{2H}}\rightarrow 0$ as $t\rightarrow s$. Due to Lemma \ref{apl-3}, $$\frac{e^{s-t}\left(e^{s-t}-1\right)
\int_{0}^\infty\int_0^{ t-s} e^{-u+v} |u-v|^{2H-2}du dv}{(t-s)^{2H}}\rightarrow 0$$ as $t\rightarrow s$. Finally,
$$\frac{e^{2s-2t}\int_0^{ t-s}\int_0^{ t-s} e^{u+v} |u-v|^{2H-2}du dv}{(t-s)^{2H}}\rightarrow \frac{1}{H(2H-1)}$$ as $t\rightarrow s$. Then it follows from the limit relations above and \eqref{differ} that
$$\frac{\mathbb{E}  (X_t-X_s)^2}{(t-s)^{2H}} \rightarrow 1$$
as $t\rightarrow s$,
therefore, there exists $d>0$ such that
\begin{equation}\label{burbur2} \mathbb{E}  (X_t-X_s)^2 \ge\frac{ 1 }{2}(t-s)^{2H}
\end{equation}
 for $t-s<d$. It follows from \eqref{burbur1} and  \eqref{burbur2}  that on some interval  we have two sided inequality \eqref{eq:(*)}. In particular, it means that assumption $(A2),(i)$ holds with $\rho=H$. Further, as always,   we are interested in values $s=ia\Delta, t=(i+1)a\Delta, 2 \le i \le \frac1a -1$. In this case we get the following relations
\begin{equation}\nonumber
\begin{gathered}
0 \le \mathbb{E}   \left(X_{(i+1)a \Delta} - X_{i a \Delta} \right)^2  \leq C \left( a^2 \Delta^2 I_0 + (a \Delta)^{2H}\right).
\end{gathered}
\end{equation}
Since $2H<2$, for $a \Delta <1$ we conclude that $(\mathbb{E}  (\xi_i^2))^2 \le C(I_0+1)^2 a^{4H} \Delta ^{4H}$, and the 1st term of the sum   $S$ (see \eqref{equ: for neg}) can be bounded as
\begin{equation}\label{1st_term}
\sum_{2 \le i \le \frac1a}(\mathbb{E}  (\xi_i^2))^2 \le C(I_0+1)^2 a^{4H-1} \Delta^{4H}.
\end{equation}

 On the other hand,   the inequality
$$a\sum_{2\le i\le 1/a} \mathbb{E}\xi_i^2 \ge 32\eta^2$$
is supplied by the inequality  $a\ge \frac{(8\eta)^{1/H}}{\Delta}$.
 Taking into account that we need to consider $a<1/2$, we can put $\Delta_*=2d$ and $\eta\le \frac{\Delta^H}{2^{3+H}}$.

Now, for any $0 \le u \le v \le s \le t$
\begin{equation}\nonumber
\begin{gathered}
\mathbb{E} (X_t-X_s)(X_v-X_u)\\= C_H \Big((e^{-t} - e^{-s})(e^{-v} - e^{-u}) \int_{-\infty}^s \int_{-\infty}^u e^{z+w}|z-w|^{2H-2}dzdw \\
+  (e^{-t} - e^{-s}) e^{-v} \int_{-\infty}^s \int_u^v e^{z+w} |z-w|^{2H-2}dzdw \\+   e^{-t} (e^{-v} - e^{-u}) \int_s^t \int_{-\infty}^u e^{z+w} |z-w|^{2H-2}dzdw\\+  e^{-t-v}   \int_s^t \int_{u}^v e^{z+w} |z-w|^{2H-2}dzdw\Big)\\=
C_H \Big((e^{s-t} - 1)(e^{u-v} - 1) \int_{\mathbb{R}_+^2 }  e^{-z-w}|z-w+s-u|^{2H-2}dzdw \\
+  (e^{s-t} - 1)   \int_{0}^\infty \int_0^{v-u} e^{-z-w} |z-w+s-v|^{2H-2}dzdw \\+     (e^{u-v} - 1) \int_0^{t-s} \int_{0}^\infty e^{-z-w} |z-w+t-u|^{2H-2}dzdw\\+      \int_0^{t-s} \int_{0}^{v-u} e^{-z-w} |z-w+t-v|^{2H-2}dzdw\Big).
\end{gathered}
\end{equation}
Taking into account that we are interested in the values $$u=ia\Delta,\, v= (i+1) a \Delta,\, s=j a \Delta,\, t=(j+1) a \Delta$$ for some $2 \le i < j \le \frac1a$, we get for $\xi_k=X_{(k+1) a \Delta} - X_{k a\Delta}, k =i,j$ that
\begin{equation}\nonumber
\mathbb{E} \xi_i \xi_j = I_{ij}^{(1)}+2I_{ij}^{(2)}+I_{ij}^{(4)},
\end{equation}
where
\begin{equation}\nonumber
\begin{gathered}
I_{ij}^{(1)}=  C_H (e^{-  a \Delta}-1)^2   \int_{\mathbb{R}_{+}^{2}} e^{-z-w} |z-w+(j-i) a \Delta|^{2H-2}dzdw,\\
I_{ij}^{(2)}=C_H  (e^{-a \Delta}-1) \int_0^{\infty} \int_0^{a \Delta} e^{-z-w} |z-w+(j-i) a \Delta|^{2H-2}dzdw,\\
I_{ij}^{(3)}= C_H \int_0^{a \Delta} \int_0^{a \Delta} e^{-z-w} |z-w+(j-i+1) a \Delta| ^{2H-2}dzdw.
\end{gathered}
\end{equation}
According to Lemma \ref{apl-2} from Appendix \ref{appn}, integral $$\int_{\mathbb{R}^2} e^{-z-w} |z-w+(j-i) a \Delta| ^{2H-2}dzdw$$ is bounded by some constant. Therefore, $I_{ij}^{(1)} \le C(a \Delta)^2$. Furthermore, according to the L'H\^{o}spital's rule
\begin{equation}\nonumber
\begin{gathered}\lim_{x\rightarrow 0} x^{-1}\int_0^{\infty} \int_0^{x} e^{-z-w} |z-w+(j-i) x|^{2H-2}dzdw\\=\lim_{x\rightarrow 0}\int_0^{\infty}e^{-x-w} |x-w+(j-i) x|^{2H-2}dzdw=\int_0^{\infty}e^{-w} w^{2H-2}dw.\end{gathered}
\end{equation}
Therefore,
$$I_{ij}^{(2)} \le C (a \Delta)^2.$$
Finally,   and under assumption that $a\Delta\rightarrow 0$
\begin{equation}\nonumber
\begin{gathered}I_{ij}^{(3)} = C_H \int_0^{a \Delta} \int_0^{a \Delta} e^{-z-w} |z-w+(j-i) a \Delta| ^{2H-2}dzdw\\
 \sim C_H \int_0^{a \Delta} \int_0^{a \Delta}   |z-w+(j-i)a \Delta| ^{2H-2}dzdw\\=C_H(a \Delta)^{2H}\int_0^{1} \int_0^{1}   |z-w+(j-i) | ^{2H-2}dzdw \sim C(a \Delta)^{2H},\end{gathered}
\end{equation}
and, according to Lemma \ref{apl-1},  $\int_0^{1} \int_0^{1}   |z-w+(j-i) | ^{2H-2}dzdw$ is bounded by some constant not depending on $j-1$. Due to the fact that $2H<2$, we can state the following: there exists   $\Delta_*>0$ such that for $\Delta<\Delta_*$, due to the fact that $a<1/2$, $\mathbb{E} \xi_i \xi_j>0$.
It means that we are exactly in conditions of Example \ref{ex-pos} and can produce the same conclusions.

\end{example}
\begin{example}\label{tempered} (Tempered fractional Brownian motion)
There are several approaches how to introduce a tempered fractional Brownian motion. For the detail see \cite{AMS, SMC, SaSu}. We shall introduce it as follows.
Let $\theta > 0$, $\alpha > 0$. Consider a process
\begin{equation}\nonumber
Y_t=\int_{-\infty}^{t} e^{-\theta (t-s)} (t-s)^{\alpha} dW_s, t \ge 0.
\end{equation}
Process $Y$ is stationary and Gaussian, with the following characteristics:
\begin{equation}\nonumber
\begin{gathered}
\mathbb{E}Y_t=0,
\mathbb{E}Y_t^2= \int_0^{\infty} e^{-2\theta z}z^{2 \alpha} dz.
\end{gathered}
\end{equation}
As usual, without loss of generality, put $\theta=1.$
Let us calculate
 \begin{equation}\nonumber
\begin{gathered}
\mathbb{E}Y_0Y_t=\mathbb{E}\int_{-\infty}^{0}e^{z} ( -z)^{\alpha}dW_z \int_{-\infty}^{t}e^{z-t} (t-z)^{\alpha}dW_z\\= \int_{-\infty}^{0}e^{z}e^{z-t}( -z)^{\alpha}(t-z)^{\alpha}dz=
 e^{-t}\int_{0}^{\infty}e^{-2z} z^{\alpha}(t+z)^{\alpha}dz\\
 = t^{2\alpha+1}e^{-t}\int_{0}^{\infty}e^{-2zt} z^{\alpha}(1+z)^{\alpha}dz\rightarrow 0
\end{gathered}
\end{equation}

as $t\rightarrow \infty$. So, $Y$
is an ergodic process. We shall not provide the small ball calculations since they are very tedious.

 \end{example}
\begin{remark} All examples are about the case where $t_0=0$. However, since we consider asymptotics of the integral, process $X$ can be arbitrary till some fixed $t_0>0$ and satisfy assumption $(A2)$  after this moment. \end{remark}
Now let us introduce more simple and stronger small ball estimate.
 \begin{itemize}
 \item[(A2), (iv)] (stronger small ball estimate)
  There exist positive constants {$\eta_*$, $\Delta_*$,} $\lambda$, $\mu$, $K_1$ and $K_2$,   such that
\beas
\sup_{s\in\R_+} \mathbb{P}\bigg[\sup_{t\in[s,s+\Delta]}
\big|X_t-X_s\big|\leq\eta\bigg]
&\leq&
K_1\exp\bigg(-K_2\eta^{-\lambda}\Delta^\mu\bigg)
\eeas
for all $\eta{ \in(0,\eta_*)}$ and $\Delta{ \in(0,\Delta_*)}$.
 \end{itemize}
 As one can see, the difference is that in $(A2), (iii)$ $\eta$ is adapted to $\Delta$, and it means that we consider $\eta$ under the curve $\eta=K_3\Delta^\gamma$, while in $(A2), (iv)$ we consider a whole rectangle $\eta{\in(0,\eta_*)}, \Delta{ \in(0,\Delta_*)}$. All previous examples do not work, however, let us consider two other  examples.
 \begin{example}\label{example-process-1} One of the simplest examples of the processes $X$ satisfying assumptions $(A2)$, is $X_t=\xi\varphi(t)$,
where $\xi$ is a random variable satisfying the following conditions:
\begin{itemize}
\item[(j)]  All moments of $\xi$ are uniformly bounded;
\item[(jj)] There exist positive constants   $\lambda_0$,   $K_3\ge 1$ and $K_4$,   such that for any $x>0$
\beas
  \mathbb{P}\left[|\xi|\leq x \right]\leq  K_3\exp\left(-K_4x^{-\lambda_0}\right),
\eeas
\end{itemize}
and function $\varphi$ is periodic with period 2, and equals $\varphi(t)=t1_{\{0\leq t\leq 1\}}+(2-t)1_{\{1\leq t\leq 2\}}.$
In this case process $X$ is continuous since $\varphi$ is continuous, and condition $(i)$ is fulfilled with $\rho=1$ because
$|\varphi(t)-\varphi(s)|\leq |t-s|$, condition $(ii)$ is supplied by $(j)$ because $\varphi$ is a bounded function. Furthermore, $$\sup_{t\in[s,s+\Delta]}
\big|X_t-X_s\big|\geq |\xi|\frac{\Delta}{2}$$
since $\sup_{t\in[s,s+\Delta]}
\big|\varphi(t)-\varphi(s)\big|\geq\frac{\Delta}{2}$, for any $0<\Delta\le 2.$ Hence
\begin{equation}\nonumber
\begin{gathered}
 \sup_{s\in\R_+} \mathbb{P}\bigg[\sup_{t\in[s,s+\Delta]}
\big|X_t-X_s\big|\leq\eta\bigg]=\sup_{s\in\R_+} \mathbb{P}\bigg[
|\xi|\frac{\Delta}{2}\leq\eta\bigg]\\
=\sup_{s\in\R_+} \mathbb{P}\left[
|\xi|\leq\frac{2\eta}{\Delta}\right]
\leq K_3\exp\left(-K_4\left(\frac{ 2\eta }{\Delta}\right)^{-\lambda_0}\right),
\end{gathered}
\end{equation}
and
so condition $(iii)$ follows from $(jj)$ with any $\eta_*>0$, $0<\Delta_*<2$, $\lambda=\mu=\lambda_0$, $K_1=K_3$, $K_2=\frac{K_4}{2^{\lambda_0}}$.
 In this case we have  a small ball estimate in time, but in some sense, uniformly  ball estimate in space. We can modify this example in the following way: let $\Omega=[0,2]$, and consider the same $\xi$ but shift the functions $\varphi$ in a random way, namely, let
$$\widetilde{\varphi}(t,\omega)=\varphi(t+\omega).$$
Then $\sup_{t\in[s,s+\Delta]}
\big|\widetilde{\varphi}(t,\omega)-\widetilde{\varphi}(s,\omega)\big|\geq\frac{\Delta}{2}$, for any $0<\Delta\le 2,$ and we have the same estimates as before.

\end{example}

\section{Divergence theorems}\label{sec_div}
The first result describes the conditions of divergence to infinity for the processes bounded in   $L^{\infty-}$. We prove it under relaxed small ball estimate $(A2),(iii)$, however, this theorem is certainly true if to replace relaxed small ball estimate with $(A2),(iv)$.
\begin{theorem}\label{theo-1} Let the function  $f$ satisfy assumption $(A1), (i)-(ii),$ and the process   $X$ satisfy assumptions  $(A2), (i)-(iii).$

Then \eqref{infdiv} holds, i.e., \beas
T^{-1+\ep}\>\bbI_T &\to& \infty\quad a.s.
\eeas
as $T\to\infty$ for every $\ep>0$.
\end{theorem}
\begin{proof}
Let $\ep>0$.
Fix positive numbers $\ep(i)$ ($i={0,}1,2,3,4$) such that
\bea\label{20181129-5}\begin{gathered}
{ \lambda\ep(0)>\mu\ep(1),\qquad \ep(0)> \gamma\ep(1), \qquad \ep(1)<\ep(2),\qquad}
 \ep(3) \><\> \frac{\ep(2)\rho}{2},\\
\qquad
{ 
K\ep(0)\><\>\ep(3)-C_0\ep(4)
}
\end{gathered}\eea

and that
\beas
{ 2K\ep(0)-}\ep(1)+\ep(2)&<&\half\ep.
\eeas
Such numbers $\ep(i)$ $(i=0,1,2,3,4)$ exist; for example, let $\delta\down0$ for
\beas\label{contr}
\ep(0)\yeq\delta^4,\quad \ep(1)\yeq\delta^5,\quad \ep(2)\yeq\delta,\quad
\ep(3)\yeq\delta^2,\quad\ep(4)\yeq\delta^3.
\eeas

 Let $\eta_n=n^{-\ep(0)}\wedge \left(K_3n^{-\gamma\ep(1)}\right)$. The 2nd value, $K_3n^{-\gamma\ep(1)}$ is necessary in order to apply assumption $(iii)$, the relaxed small ball estimate. However, we shall deal with the 1st value, $n^{-\ep(0)}$, therefore, assume that $n$ is sufficiently large, such that $$\log n> \frac{\log \left(\frac{1}{K_3}\right)}{\epsilon(0)-\gamma\epsilon(1)}.$$  In this case $n^{-\ep(0)}< K_3n^{-\gamma\ep(1)}$, and
 $\eta_n=n^{-\ep(0)}$. Then
\beas
{  c_n}&:=& \frac{1}{4}\min\bigg\{
\inf_{x\in\R}\sup_{y\in H_+(x,{ \eta_n})}\big|f(y)\big|,\>
\inf_{x\in\R}\sup_{y\in H_-(x,{ \eta_n})}\big|f(y)\big|
\bigg\}
 \>\geq\>{  \frac{1}{4}\eta_n^K}
\eeas
for large $n$.
Let
$\Delta_n=n^{{ -\ep(1)}}$,
$s^n_j=(j-1)\Delta_n$ and $t^n_j=j\Delta_n$ for $j,n\in\bbN$.
Let $I^n_j=[s^n_j,{  s^n_j+\Delta_n/2}]$.
%
Let
\beas
A^n_j &=&
\bigg\{\sup_{s,t\in I^n_j}\big|X_t-X_s\big|>{ \eta_n}\bigg\}.
\eeas
For $\omega\in A^n_j$, there exist $\tau(\omega)$, $\sigma(\omega)\in I^n_j$ such that
$\sigma(\omega)<\tau(\omega)$ and that
$\big|X_{\tau(\omega)}(\omega)-X_{\sigma(\omega)}(\omega)\big|>{\eta_n}$.
Therefore, by the mean-value theorem, 
if $X_{\sigma(\omega)}(\omega)<X_{\tau(\omega)}(\omega)$, then
\beas
\max_{t\in I^n_j}|f(X_t(\omega))|
&\geq&
\sup_{t\in[\sigma(\omega),\tau(\omega)]}|f(X_t(\omega))|
\ygeq
\sup_{x\in{\sf H}_+(X_{\sigma(\omega)},\eta_n)}\big|f(x)\big|
\ygeq
4{ c_n}.
\eeas
Similarly, if $X_{\tau(\omega)}(\omega)<X_{\sigma(\omega)}(\omega)$, then we consider ${\sf H}_-(X_{\sigma(\omega)},\eta_n)$ and conclude that
\bea\label{20181128-1}
\max_{t\in I^n_j}|f(X_t(\omega))|\geq4{ c_n}.
\eea
Thus, inequality (\ref{20181128-1}) is always valid for $\omega\in A^n_j$.

Let $\beta=\frac{2\ep(3)}{\ep(2)}$ and let $r>(\rho-\beta)^{-1}$, equivalently,
$\rho-\frac{1}{r}>\beta$.
By $(A2)$ (i),
\beas
\mathbb{E}\big[|X_t-X_s|^r\big]
&\leq&
B(r)|t-s|^{r\rho}\qquad(t\in[s,s+1])
\eeas
where
\beas
B(r)=\bigg(\sup_{s,\; t\in\R_+:\>s<t<s+1}\frac{\big\|X_t-X_s\big\|_r}{|t-s|^{\rho}}\bigg)^r.
\eeas
Then by the Garsia-Rodemich-Ramsey  inequality, there exists a constant $C(r)$ (independent of $s$) such that
\bea\label{20181129-6}
\mathbb{P}\bigg[\sup_{t_1,\;t_2\in[s,s+1]:\>t_1\not=t_2}
\frac{|X_{t_2}-X_{t_1}|}{|t_2-t_1|^\beta}\geq h\bigg]
&\leq&
\frac{C(r)B(r)M}{h^r}
\eea
for all $h>0$, where
\begin{equation}\label{MM}
M  =
\int_{[0,1]}\int_{[0,1]}|t_2-t_1|^{r\rho-r\beta-2}dt_1dt_2;
\end{equation}
$M$ is finite since $r\rho-r\beta-2>-1$.
Since
\beas
\sup_{t\in [s,s+n^{-\ep(2)}]}\frac{|X_t-X_s|}{|t-s|^\beta}
&\geq&
n^{\ep(2)\beta}\sup_{t\in [s,s+n^{-\ep(2)}]}|X_t-X_s|,
\eeas
by setting $h=n^{\ep(3)}$ in (\ref{20181129-6}), and taking into account the definition of $\beta$, we obtain
\bea\begin{gathered}\label{20181129-7}
\sup_{s\in\R_+}\mathbb{P}\bigg[\sup_{t\in [s,s+n^{-\ep(2)}]}|X_t-X_s|\geq n^{-\ep(3)}\bigg]
 \\ \leq
\sup_{s\in\R_+}\mathbb{P}\bigg[\sup_{t\in [s,s+n^{-\ep(2)}]}\frac{|X_t-X_s|}{|t-s|^\beta}\geq n^{-\ep(3)+\ep(2)\beta}\bigg]
 \\    =
 \sup_{s\in\R_+}\mathbb{P}\bigg[\sup_{t\in [s,s+n^{-\ep(2)}]}\frac{|X_t-X_s|}{|t-s|^\beta}\geq n^{\ep(3)}\bigg]\leq
\frac{C(r)B(r)M}{n^{\ep(3)r}}
\end{gathered}\eea
for all $n\in\bbN$.

\begin{en-text}
\beas
f(y)-f(x)
&=&
\int_0^1 \partial_x f ((1-s)x+sy)ds(y-x)
\eeas
Let
\beas
\xi^n_j &=& \inf\big\{t\geq \min I^n_j;\>\big|f(X_t)\big|\geq2c\big\}
\eeas
\end{en-text}
%

Obviously,
\begin{equation}\begin{gathered}\label{201812060238}
\mathbb{P}\left[\sup_{t\in[s,s+n^{-\ep(2)}]}\big|f(X_t)-f(X_{s})\big|\geq { c_n}\right]
 \leq
\mathbb{P}\left[|X_s|\geq n^{\ep(4)}\right]
 \\
+\mathbb{P}\bigg[|X_s|\leq n^{\ep(4)},\>\sup_{t\in [s,s+n^{-\ep(2)}]}|X_t-X_s|\leq n^{-\ep(3)},
 \\  \sup_{t\in[s,s+n^{-\ep(2)}]}\big|f(X_t)-f(X_{s})\big|\geq { c_n}\bigg]
 \\
+\mathbb{P}\left[\sup_{t\in [s,s+n^{-\ep(2)}]}|X_t-X_s|\geq n^{-\ep(3)}\right].
\end{gathered}\end{equation}
By choosing a sufficiently large $r$ in (\ref{20181129-7}), we know
\bea\label{20181129-8}
\sup_{s\in\R_+}\mathbb{P}\left[\sup_{t\in [s,s+n^{-\ep(2)}]}|X_t-X_s|\geq n^{-\ep(3)}\right]
&=&
O(n^{-L})
\eea
as $n\to\infty$ for every $L>0$.
Moreover, from $(A2), (ii)$ we have
\bea\label{20181129-8}
\sup_{s\in\R_+}\mathbb{P}\big[|X_s|\geq n^{\ep(4)}\big]
&=&
O(n^{-L})
\eea
as $n\to\infty$ for every $L>0$.
By Taylor's formula applied to $f$, we obtain
\beas
  \sup_{t\in[s,s+n^{-\ep(2)}]} \big|f(X_t)-f(X_{s})\big|
&\leq&
\sup_{t\in[s,s+n^{-\ep(2)}]}C_0\big(1+|X_s|+|X_t-X_s|\big)^{C_0}\big|X_t-X_s\big|
\\&\leq&
C_0(2+n^{\ep(4)})^{C_0}n^{-\ep(3)}
\eeas
whenever
$|X_s|\leq n^{\ep(4)}$ and $\>\sup_{t\in [s,s+n^{-\ep(2)}]}|X_t-X_s|\leq n^{-\ep(3)}$.
We also have $$C_0(2+n^{\ep(4)})^{C_0}n^{-\ep(3)}<{ n^{-K\ep(0)}/4=\eta_n^K/4\leq c_n}$$ for large $n$.
Therefore, under $(A1), (ii)$ and $(A2), (i)-(ii)$,
\bea\label{20181129-3}
\sup_{s\in\R_+}
\mathbb{P}\bigg[\sup_{t\in[s,s+n^{-\ep(2)}]}\big|f(X_t)-f(X_{s})\big|\geq { c_n}\bigg]
&=&
O(n^{-L})
\eea
as $n\to\infty$ for every $L>0$. 
\begin{en-text}
We remark that inequality (\ref{20181129-3}) holds also under $(A1^\sharp), (ii)$ and $(A2), (i)$;
proof is similar but without cutting with the event $\{|X_s|\geq n^{\ep(4)}\}$.
\end{en-text}

Let
\beas
\tau^n_j=\inf\{t\geq s^n_j;\>\big| f(X_t)\big|\geq 4c_n\}\wedge ({ s^n_j+\Delta_n/2}),
\eeas
and let
\bea\label{events-B}
B^n_j
&=&
\left\{\sup_{t\in[\tau^n_j,\tau^n_j+n^{-\ep(2)}]}\big|f(X_t)-f(X_{\tau^n_j})\big|< 3{ c_n}\right\}.
\eea
Let $J_n=\lceil n^{1{+}\ep(1)}\rceil+1$.
Now (\ref{20181129-3}) gives the estimate
\begin{equation}\begin{gathered}\label{20181129-9}
\mathbb{P}\bigg[\bigg(\bigcap_{j=1}^{J_n}B^n_j\bigg)^c\>\bigg]
 \\\leq
3\sum_{j=1}^{J_n}\sum_{i=1}^{\lceil n^{{-}\ep(1)+\ep(2)}\rceil+1}
\mathbb{P}\bigg[\sup_{t\in[s^n_j+(i-1)n^{-\ep(2)},s^n_j+in^{-\ep(2)}]}
\big|f(X_t)
 \\-
f(X_{s^n_j+(i-1)n^{-\ep(2)}})\big|\geq { c_n}\bigg]
=
O(n^{-L})
\end{gathered}\end{equation}
as $n\to\infty$ for every $L>0$.

By $(A2), (iii)$ (recall that $\eta_n<K_3\Delta_n^\gamma$), we have
\beas
\sup_{j\leq J_n} \mathbb{P}\left[\sup_{s,t\in I^n_j}
\left|X_t-X_s\right|\leq{\eta_n}\right]
&=&
O(n^{-L})
\eeas
as $n\to\infty$ for every $L>0$.
That is,
\bea\label{20181129-12}
\mathbb{P}\left[\left(\bigcap_{j\leq J_n}A^n_j\right)^c\>\right] &=& O(n^{-L})
\eea
as $n\to\infty$ for every $L>0$.

 On $A^n_j\cap B^n_j$ we have that $|f(X_t)|\geq c_n$ on the interval of length at least $n^{-\ep(2)}$, therefore
\beas
\int_{[s^n_j,t^n_j]}f(X_t)^2dt
&\geq&
{ c_n}^2n^{-\ep(2)}.
\eeas
Therefore,
\bea\label{20181129-11}
\bbI_n
&\geq&
{ \frac{1}{16}n^{-2K\ep(0)}}\lfloor n^{1{+}\ep(1)}\rfloor n^{-\ep(2)}
\eea
on $\bigcap_{j=1}^{J_{n}} \big(A^n_j\cap B^n_j\big)$.
Thanks to (\ref{20181129-9}), (\ref{20181129-12}) and (\ref{20181129-11})
with the inequality ${ 2K\ep(0)-}\ep(1)+\ep(2)<\ep/2$, we obtain
\beas
\mathbb{P}\left[\bbI_n<n^{1-\half\ep}\right] &=& O(n^{-L})
\eeas
as $n\to\infty$ for every $L>0$.
In particular, by Borel-Cantelli's lemma,
\beas
\mathbb{P}\bigg[\limsup_{n\to\infty}\big\{\bbI_n<n^{1-\half\ep}\big\}\bigg] &=& 0.
\eeas
Therefore,
\bea\label{20181205-1}
\mathbb{P}\left[\lim_{n\to\infty}\big(n^{-(1-\ep)}\bbI_n\big)=\infty\right] &=& 1.
\eea
This shows \eqref{infdiv} for $T=n$ but it is sufficient for proof of the theorem.
\end{proof}

\begin{theorem}\label{theo-1.2}
 The convergence \eqref{infdiv} holds
 if we exclude condition $(A2), (ii)$ of boundedness in $L^{\infty-}$, and instead add the condition

$(A1) (iii)$  There exist constants $Q>0$, $p>0$ and $C>0$ such that
$$|f(x)|\ge Q|x|^p$$
for any $|x|\ge C$.
\end{theorem}
\begin{proof}  Analyzing proof of Theorem \ref{theo-1}, we can see that condition $(A2), (ii)$ is applied only when we construct the upper bound for the first term
  on  the right-hand side of the inequality  \eqref{201812060238}. In this connection, we can  consider,  instead of the first two terms  on  the left-hand side of \eqref{201812060238}, one term    of the form
\bea\label{201812060228}
 \sup_{s\in\R_+}P\bigg[\sup_{t\in [s,s+n^{-\ep(2)}]}|X_t-X_s|\geq n^{-\ep(3)},\>
 |X_s|\leq n^{\ep(4)}\bigg]
 &=&
 O(n^{-L})
 \eea
It means that instead of \eqref{events-B} we should consider the events
\beas
\widetilde{B}^n_j
&=&
\bigg\{\sup_{t\in[\tau^n_j,\tau^n_j+n^{-\ep(2)}]}\big|f(X_t)-f(X_{\tau^n_j})\big|< 3{ c_n}\bigg\}\bigcup \bigg\{|X_{\tau^n_j}|\geq n^{\ep(4)}\bigg\}.
\eeas
Then the upper bound \eqref{20181129-9} still holds
  for $\widetilde{B}^n_j$ in place of ${B}^n_j$
, and moreover, on $A^n_j\cap \widetilde{B}^n_j$ we have, as before that  $|f(X_t)|\geq c_n$ on the interval of length at least $n^{-\ep(2)}$, otherwise, for sufficiently large $n$
 $$| X_t|\geq n^{\ep(4)}- n^{-\ep(3)}\ge 1/2n^{\ep(4)},$$
   whence  $$|f( X_t)|\geq \frac{Q}{2^p}n^{p\ep(4)},$$
therefore,  for sufficiently large $n$
\beas
\int_{[s^n_j,t^n_j]}f(X_t)^2dt
&\geq&
{ c_n}^2n^{-\ep(2)}\wedge \frac{Q}{2^p}n^{p\ep(4)- \epsilon(1)}{\colorred,}
\eeas
and we can conclude as in Theorem \ref{theo-1}.
\end{proof}

Note that without condition $(A2), (ii)$, the diverging rate obtained here could be far from optimal. This is confirmed by the following statement.
\begin{theorem} Let the function $f(x)=|x|^p, p>1$, and   let the process  $X=\{X_t, t\geq 0\}$ be a real-valued stochastic process, satisfying the following conditions
\begin{itemize}
\item[$(i)$] $X$ is self-similar with index $H\in(0,1)$;
\item[$(ii)$] The random variable $\int_0^1|X_t|^pdt$ satisfies assumption
 $\int_0^1|X_t|^pdt\ge \xi,$ where $\xi$ is a non-negative random variable with bounded density (particularly, $\int_0^1|X_t|^pdt$ itself has a bounded density).
 Then for any $\epsilon\in(0, pH)$ we have that
 \beas
\liminf_{T\to\infty}T^{-1-\ep}\int_0^T |X_t|^p dt &>& 0\qquad a.s.
\eeas
\end{itemize}
 \end{theorem}
 \begin{proof} Let constant $C>0$ is an upper bound for the density of the random variable $\xi$ from assumption $(ii)$. Then  for any $k\in \mathbb{N}$, $0<\epsilon<pH$, $\beta>0$ and $x>0$ it follows from the self-similarity of the finite-dimensional distributions of $X$ that
\begin{equation*}\begin{gathered}
P_{k,x} := \mathbb{P} \left \{\frac{\int_0^{k^{\beta}} |X_t|^p dt}{k^{\beta(1+\epsilon)}} <x\right \}=
\mathbb{P} \left \{\frac{\int_0^1 |X_{sk^{\beta}}|^p ds}{k^{\beta\epsilon}} <x\right \}\\=
\mathbb{P} \left \{k ^{\beta \left(pH-\epsilon \right)} \int_0^1  |X_s|^p ds  <x \right \} =
\mathbb{P} \left \{\int_0^1  |X_s|^p ds  < \frac{x}{k^{\beta\left(pH-\epsilon \right)}} \right \}\\ \leq
\mathbb{P} \left \{  \xi  < \frac{x^{\frac1p}}{k^{\beta \left(H- \frac{\epsilon}{p}\right)}} \right \}
\leq C\frac{x^{\frac1p}}{k^{\beta \left(H- \frac{\epsilon}{p}\right)}}.
\end{gathered}\end{equation*}

If we choose $\beta > \left( H- \frac \epsilon p \right) ^{-1}$, then $\Sigma_{k \geq 1}P_{k,x}$
converges, and it follows from Borel-Cantelli and the fact that $x>0$ is arbitrary that
\begin{equation*}
\lim_{k\to\infty}  \frac{\int_0^{k^{\beta}} |X_t|^p dt } {k^{\beta\left(1+\epsilon\right)}}=\infty\ \text{a.s.}
\end{equation*}
Further, for any $T \in \left[k^{\beta},(k+1)^\beta) \right]$
\begin{equation*}\begin{gathered}
\frac{\int_0^T |X_t|^p dt}{T^{1+\epsilon}} \geq \frac{\int_0^{k^{\beta}} |X_t|^pdt}{k^{\beta \left(1+\epsilon \right)}} \left( \frac{k}{k+1}\right)^{\beta \left( 1+ \epsilon\right)} \geq \frac{1}{2^{\beta \left(1+\epsilon\right)}} \frac{\int_0^{k^{\beta}}|X_t|^pdt}{k^{\beta \left(1+\epsilon \right)}},
\end{gathered}\end{equation*}
therefore
\begin{equation*}
\liminf_{T\to\infty}\frac{\int_0^T |X_t|^p dt}{T^{1+\epsilon}} = +\infty  \ \text{a.s.}
\end{equation*}
for any $0 < \epsilon <pH$. \end{proof}
{
\begin{example}\rm

For example, for any $p>1$
\beas
\liminf_{T\to\infty}T^{-1-\ep}\int_0^T |B_t^H|^p dt &>& 0\qquad a.s.
\eeas
for any $0<\epsilon<pH$, where $B^H$ is a fractional Brownian motion with Hurst index $H\in(0,1)$. Indeed, $B^H$ is a self-similar process with the index $H$ of self-similarity, and in this case   $$\int_0^1|B_t^H|^pdt\ge  \xi  = |\mathcal{N}(0, \sigma^2)|^p,$$ and
\begin{equation*}
\sigma^2=\frac12 \int_0^1 \int_0^1\left(s^{2H}+u^{2H}-|s-u|^{2H} \right) du ds= \frac{1}{2H+2}.
\end{equation*}
Obviously, $\xi$ has a bounded density. However, in this particular case we can say more and establish the exact rate of convergence. Indeed, consider $\epsilon= pH$. In this case, according to Theorem 3.3 \cite{MR}.
\begin{equation*}
\liminf_{T\to\infty} \frac{\sup_{0 \leq s \leq T } |B_s^H|^p \left(\log \log T \right)^{pH}}{T^{pH}} = c > 0 \ \text{a.s.},
\end{equation*}Therefore,
\begin{equation*}
P_{k,x} \leq 2 \frac{1}{\sigma \sqrt{2 \pi}} \frac{x^{\frac1p}}{k^{\beta \left(H-\frac{\epsilon}{p}\right)}}.
\end{equation*}
where $c$ is a positive constant. Therefore,
\begin{equation*}\begin{gathered}
\liminf_{T\to\infty}\frac{\int_0^T |B_s^H|^p ds}{T^{1+pH}} \leq \liminf_{T\to\infty} \frac{\sup_{0 \leq s \leq t } |B_s^H|^p}{T^{pH}} =0  \ \text{a.s.}
\end{gathered}\end{equation*}
A fortiori, for any $\epsilon > pH$
\begin{equation*}
\liminf_{T\to\infty} \frac{\int_0^T |B_s^H|^p ds}  {T^{\left(1+\epsilon\right)}}=0  \ \text{a.s.}
\end{equation*}

\end{example}

%
%

%
\

\halflineskip
Concluding this section, we will consider a stationary $X$.
Let $g_t=f(X_t)^2$.
Suppose that $\{g_t\}_{t\in\bbR_+}$ is uniformly integrable,
which is satisfied,  for example, under condition $(A2)$, $(ii)$, and that process $X$ is stationary.
Then
\bea
\frac{1}{T}\int_0^T g_t dt
&\to&
Z
\quad a.s.
\eea
as $T\to\infty$ for some nonnegative random variable $Z$
by Birkhoff's individual ergodic theorem;
$Z$ is a random variable measurable to the invariant $\sigma$-field.
Define the event $A$ by
\bea\label{202101311231}
A
&=&
\big\{Z=0
\big\}.
\eea
The family $\big\{T^{-1}\int_0^Tg_tdt\big\}_{T>0}$ of random variables
is uniformly integrable, and hence
\bea\label{202101311230}&&
\lim_{T\to\infty}\frac{1}{T}\int_0^T E[g_t 1_A]dt
\yeq
\lim_{T\to\infty}E\bigg[\frac{1}{T}\int_0^T g_t dt1_A\bigg]
\nn\\&=&
E\bigg[\lim_{T\to\infty}\frac{1}{T}\int_0^T g_t dt1_A\bigg]
\yeq
E[Z1_A]
\yeq 0.
\eea
\begin{en-text}
Therefore
\bea\label{202101311233}
\liminf_{t\in\bbR_+}\int_t^{t+h}E\big[g_t1_A\big] dt&=& 0
\eea
for any $h>0$.
\end{en-text}
Since
$(g_s,1_A)=^d(g_t,1_A)$ for any $s,t\in\bbR_+$ by stationarity of $X$,
(\ref{202101311230}) implies
\bea\label{202101311234}
E\big[g_t1_A\big] &=& 0
\eea
for any $t\in\bbR_+$, and hence,
\bea\label{202101311321}
E\bigg[\int_0^Tg_tdt1_A\bigg] &=& 0
\eea
for any $T\in\bbR_+$.
On the other hand, if
\bea\label{202101311322}
\sup_{T>0}
\int_0^Tg_tdt &>& 0\quad  a.s.,
\eea
then (\ref{202101311321}) is valid only when $P(A)=0$, i.e.,
$Z>0$ a.s. and then $\int_0^Tg_tdt$ diverges at the rate of $T$ a.s.
as $T\to\infty$.
In particular, under the conditions of Theorem \ref{theo-1},
it holds that
\beas
\lim_{T\to\infty}\>\frac{1}{T}\int_0^T g_tdt &>& 0\quad a.s.
\eeas
This supplements Theorem \ref{theo-1}'s result on the divergence of the integral.

\section{Statistical Application}\label{stat_app}

Let us consider two   statistical applications of the divergence results. Namely, let us  consider  Ornstein--Uhlenbeck process $Y=\lbrace  Y_t, t \ge 0 \rbrace$ with unknown drift parameter  $\theta > 0$    that is the solution of the equation
\begin{equation}\nonumber
Y_t=Y_0 - \theta \int_0^t Y_s ds +\int_0^t g_s dW_s,
\end{equation}
where $W=\lbrace W_t, t \ge 0 \rbrace$ is a Wiener process, $g: \mathbb{R}_+ \rightarrow \mathbb{R}$ is a measurable function such that $0 \le c\le |g_s| \le C,\; s\ge 0$. Since $Y$ satisfies assumption $(A2)$ and $f(x)=x^2$ satisfies assumption $(A1)$, we can conclude that both $T^{-1+\epsilon} \int_0^T Y_s^2 ds  \rightarrow \infty$ and $T^{-1+\epsilon} \int_0^T Y_s^2 g_s^2 ds  \rightarrow \infty$ for any $\epsilon>0$ as $T \rightarrow \infty$. Consider the equality
\begin{equation}\nonumber
\int_0^T Y_t d Y_t  = - \theta \int_0^T Y_s^2 ds + \int_0^T Y_s g_s dW_s,
\end{equation}
whence
\begin{equation}\nonumber
\frac{\int_0^T Y_t d Y_t }{\int_0^T Y_s^2 ds} = - \theta + \frac{\int_0^T Y_s g_s dW_s}{\int_0^T Y_s^2 ds}.
\end{equation}
Furthermore, since $ \int_0^T Y_s^2 g_s^2 ds  \rightarrow \infty$, we get from the strong law of large numbers  for martingales that
\begin{equation*}
\frac{\int_0^T Y_s g_s dW_s}{\int_0^T Y_s^2 g_s^2 ds} \rightarrow 0 \ \text{a.s. as } T  \rightarrow \infty.
\end{equation*}
However,  $$ c^2 \le \frac{\int_0^T Y_s^2 g_s^2 ds}{\int_0^T Y_s^2 ds} \le C^2,$$ and it means that
\begin{equation*}
\frac{\int_0^T Y_s g_s dW_s}{\int_0^T Y_s^2 ds} \rightarrow 0 \ \text{a.s. as } T  \rightarrow \infty.
\end{equation*}
We get that $-\frac{\int_0^T Y_s d Y_s}{\int_0^T Y_s^2 ds}$ is a strongly consistent estimator of $\theta$.

Another example can be introduced as follows. Let the processes $X$ and $Y$ be observable, and satisfy the relation
 $$X_t=X_0+\theta\int_0^tg(Y_s)ds+B_t^H,$$
 without any restriction on $\theta\in \R$ and $H\in(0,1)$, but assuming that $g=f^2$, where $f$ satisfies $(A1)$ and $Y$
 satisfies $(A2)$. Then $$\frac{X_T}{\int_0^Tg(Y_s)ds}=\frac{X_0}{\int_0^Tg(Y_s)ds}+\theta+\frac{B_T^H}{\int_0^Tg(Y_s)ds}.$$ According to \cite{KMM},
 $$\frac{B_T^H}{T^{H+\epsilon}}\rightarrow 0$$ a.s. as $T\rightarrow \infty$ for any $\epsilon>0$ while $\frac{\int_0^Tg(Y_s)ds}{T^{H+\epsilon}}\rightarrow \infty$ a.s. for $H+\epsilon\le 1$.
\section{Appendix}\label{appn}
Here we establish some auxiliary results.
\begin{lemma}\label{apl-1}
Let $H \in (1/2,1), x ,y \ge 0$. Then there exists $C>0$ depending only on H such that
\begin{equation}\nonumber
\int_0^x |w-y|^{2H-2}dw \le C x ^{2H-1}.
\end{equation}
\end{lemma}
\begin{proof}
We consider only $x>0$. Let $y=0$. Then $\int_0^x w^{2H-2}dw = (2H-1)^{-1}x^{2H-1}$. \\
Let $0 < y \le x$. Then
\begin{equation}\nonumber
\begin{gathered}
\int_0^x |w-y|^{2H-2}dw =\int_0^y (y-w)^{2H-2}dw+\int_y^x (w-y)^{2H-2} dw \\
=(2H-1)^{-1}  \left( y^{2H-1} +(x-y)^{2H-1} \right) \le 2 (2H-1)^{-1} x^{2H-1}.
\end{gathered}
\end{equation}
Let $ y \ge x$. Then, since $|a^{\alpha}-b^{\alpha}| \le |a-b|^{\alpha}$ for $\alpha \in (0,1)$, we have that
\begin{equation}\nonumber
\begin{gathered}
\int_0^x |w-y|^{2H-2}dw = \int_0^x (y-w)^{2H-2}dw\\
 =2H-1)^{-1} \big[ y^{2H-1}-(y-x)^{2H-1} \big] \le (2H-1)^{-1} x^{2H-1}.
\end{gathered}
\end{equation}
Lemma is proved.
\end{proof}
\begin{lemma} \label{apl-2}
There exists  such $C>0$ that for any $p>0$
\begin{equation}\nonumber
I_5 :  =  \int_{\mathbb{R}^2_{+}} e^{-z-w} |z-w+p|^{2H-2} dzdw \le C.
\end{equation}
\end{lemma}
\begin{proof} Let us provide the following transformations:
\begin{equation}\nonumber
\begin{gathered}
I_5 = \int_0^{\infty} e^{-z}  \int_0^{z+p}  e^{-w} (z+p-w)^{2H-2}dwdz+
\int_0^{\infty} e^{-z} \int_{z+p}^{\infty} e^{-w} (w-z-p)^{2H-2}dwdz \\
 =\int_0^{\infty} e^{-z} \int_0^{z+p}  e^{x-z-p} x^{2H-2}dxdz+
\int_p^{\infty} e^{-w} \int_0^{w-p} e^{-z} (w-p-z)^{2H-2}dzdw\\
=e^{-p} \left( \int_0^p e^x x^{2H-2} dx \int_0^{\infty} e^{-2z}dz+
\int_p^{\infty}e^x x^{2H-2} \int_{x-p}^{\infty} e^{-2z}dzdx \right)\\
+\int_p^{\infty} e^{-w} \int_0^{w-p} e^{p-w+x}x^{2H-2}dxdw =:I_6+I_7+I_8.
\end{gathered}
\end{equation}
Since $$\lim_{p \rightarrow \infty} e^{-p} \int_0^p e^x x^{2H-2}dx = \lim_{p \rightarrow \infty}  \frac{e^p p^{2H-2}}{e^p} =0,$$ the value $I_6 = e^{-p} \int_0^p e^x x^{2H-2}dx$ is bounded. Further,
\begin{equation}\nonumber
\begin{gathered}
I_7 = e^{-p} \int_p^{\infty} e^x x^{2H-2} \int_{x-p}^{\infty} e^{-2z}dzdx\\
= \frac12 e^{-p} \int_p^{\infty} e^x e^{-2x+2p}x^{2H-2}dx =\frac12 e^p \int_p^{\infty} e^{-x}x^{2H-2}dx,
\end{gathered}
\end{equation}
and
\begin{equation}\nonumber
\begin{gathered}
\lim_{p \rightarrow \infty} \frac{\int_p^{\infty}e^{-x}x^{2H-2}dx}{e^{-p}} = \lim_{p \rightarrow \infty} \frac{e^{-p}p^{2H-2}}{e^{-p}} = 0,
\end{gathered}
\end{equation}
therefore this value is bounded, too.
Finally,
\begin{equation}\nonumber
\begin{gathered}
I_8= e^p \int_p^{\infty} e^{-2w} \int_0^{w-p} e^x x^{2H-2} dxdw \le (2H-1)^{-1} \int_p^{\infty} e^{-w} (w-p)^{2H-1} dw \\
= (w-p=x)=\\
= (2H-1)^{-1} e^{-p} \int_0^{\infty} e^{-x} x^{2H-1} dx \rightarrow 0, p \rightarrow  \infty.
\end{gathered}
\end{equation}
Therefore, this values is bounded, too. Lemma is proved.
\end{proof}

\begin{lemma} \label{apl-3} For any $H\in(1/2, 1)$ we have the limit relation
$$\lim_{x\rightarrow 0}x^{-1}\int_0^{\infty}\int_0^{x} e^{-u+v}  |u-v|^{2H-2} dudv=\Gamma(2H-1). $$
\end{lemma}
\begin{proof} Applying L'Hospital's rule, we immediately get that
\begin{equation}\nonumber
\begin{gathered}
\lim_{x\rightarrow 0}x^{-1}\int_0^{\infty}\int_0^{x} e^{-u+v}  |u-v|^{2H-2} dudv
=\lim_{x\rightarrow 0}\int_0^{\infty}e^{-u+x}  |u-x|^{2H-2} du\\ =
\lim_{x\rightarrow 0}\int_{-x}^{\infty}e^{-z}  |z|^{2H-2} du=\Gamma(2H-1),
\end{gathered}
\end{equation}
and lemma is proved.
\end{proof}

\bibliographystyle{elsarticle-num}

\end{document}